\documentclass[12pt]{amsart}
\usepackage[T1]{fontenc}

\usepackage[cp1250]{inputenc}
\usepackage{graphicx}
\usepackage{color}
\usepackage{amsfonts,amsthm,amssymb,mathrsfs,cite} 
\allowdisplaybreaks
\newtheorem{thm}{Theorem}[section]

\newtheorem{lemma}[thm]{Lemma}

\newtheorem{definition}[thm]{Definition}

\theoremstyle{remark}

\numberwithin{equation}{section}

\newcommand{\eqnum}{\refstepcounter{equation}\textup{\tagform@{\theequation}}}

\begin{document}

\title[Stochastic hKdVB type equation]
{Martingale solution of stochastic hybrid Korteweg - de Vries - Burgers  equation}

\author[Karczewska]{Anna Karczewska}
\address{Faculty of Mathematics, Computer Science and Econometrics\\ University of Zielona G\'ora, Szafrana 4a, 65-516 Zielona G\'ora, Poland}
 \email{a.karczewska@wmie.uz.zgora.pl} \thanks{}

\author[Szczeci\'nski]{Maciej Szczeci\'nski}
\address{Faculty of Management, Computer Science and Finance\\ Wroc\l{}aw University of Economics, Komandorska 118/120, 53-345 Wroc\l{}aw, Poland }
\email{maciej.szczecinski@ue.wroc.pl}

\date{\today}




\begin{abstract}
In the paper, we consider a stochastic hybrid Korteweg - de Vries - Burgers type equation with multiplicative noise in the form of cylindrical Wiener process. We prove the existence of a martingale solution to the equation studied. 
The proof of the existence of the solution is based on two approximations of the considered problem and the compactness method. First, we introduce an auxiliary problem corresponding to the equation studied. Then, we prove the existence of a martingale solution to this problem. Finally, we show that the solution of the auxiliary problem converges, in some sense, to the solution of the equation under consideration.

\end{abstract}

\maketitle	

\section{Introduction} \label{int1}

The deterministc hybrid  Korteweg - de Vries - Burgers (hKdVB for short) equation has been derived by Misra, Adhikary and Shuka \cite{MAS12} and Elkamash and Kourakis \cite{ElkaK} in the context of  shock excitations in multicomponent plasma. The hKdVB equation, derived in stretched coordinates 
$\xi=\epsilon^{\frac{1}{2}}(x-vt),~ \tau=\epsilon^{\frac{3}{2}}t$ ($v$ is the phase velocity of the wave) has the form
\begin{equation} \label{hkdvb}
u_{\tau} +A u\,u_{\xi}+Bu_{3\xi}=Cu_{2\xi}-Du.
\end{equation}
In (\ref{hkdvb}), $u(\xi,\tau)$ represents electrostatic potential or electric field pulse in the reference frame moving with the velocity $v$. Indexes denote partial derivatives, that is, $u_{\tau}=\partial u/\partial\tau$, \linebreak $u_{2\xi}=\partial^2u/\partial\xi^2$ and so on. Constants $A,B,C,D$ are related to parameters describing properties of plasma \cite[Eq.~(27)]{ElkaK}. 

Although the equation (\ref{hkdvb}) was derived for dissipative dispersive waves in multicomponent plasma, it can be applied in several other physical systems, e.g., surface water waves and the motion of optical impulses in fibers.
For some particular values of constants $A,B,C,D$ the hKdVB equation (\ref{hkdvb}) reduces to the particular cases:
\begin{itemize} 
\item the Korteweg - de Vries equation, when $C=D=0$;
\item the damped (dissipative) KdV equation, when $C=0$;
\item the Burgers equation, when $B=D=0$;
\item the KdV-Burgers equation, when $D=0$;
\item the damped Burgers equation, when $B=0$.
\end{itemize}
The term with $A\ne0$ introduces nonlinearity, that with $B\ne0$ is responsible for dispersion, $C\ne0$ supplies diffusive term and $D\ne0$ introduces damping.
All equations of these kinds were widely studied 30-40 years ago, and most of physical ideas have been already understood (see, e.g., Lev Ostrovsky's book \cite{LOst} and references therein). On the other hand, during the last few years, one can observe renewal of interest in this fields, mostly due to extensions to higher order equations. 

Studies of the full generalized hybrid KdB-Burgers equation (\ref{hkdvb}) have appeared in the physical literature only in \cite{ElkaK,MAS12}. Some approximate analytic solutions and several cases of numerical solutions to  (\ref{hkdvb}) were subjects of recent studies in \cite{ElkaKVCRK}. 


The paper deals with a stochastic hybrid Korteweg - de Vries - Burgers type equation.
The presence of stochastic noise has deep physical grounds. In the case of waves in plasma, it can be caused by thermal fluctuations, whereas in the case of water surface waves by air pressure fluctuations due to the wind. To the best of our knowledge, our  
paper is the first one which deals with the stochastic hKdVB equation. 

The main result of the paper, Theorem \ref{P4.1} supplies the existence of a  martingale solution to the equation (\ref{Apr}), which is the stochastic version of the equation (\ref{hkdvb}). 

The idea of the proof of the existence of a martingale solution to (\ref{Apr}) is the following. First, we introduce an auxiliary problem (\ref{par}) which we can call $\varepsilon$-approximation of the equation  (\ref{Apr}). Then, in Lemma \ref{parMart} we prove that the problem (\ref{par}) has a martingale solution. Here we use the Galerkin approximation (\ref{Galerkin}) of (\ref{par}) and the tightness of the family of distributions of the solutions to the approximation (\ref{Galerkin}). Next, in Lemma \ref{szac4.1} we show two estimates used in the proof of Theorem \ref{P4.1}. Lemma \ref{PropCias} guarantees the tightness of the family of distributions of solutions to the problem (\ref{par}) in a proper space. Finally, we prove that the solution to(\ref{par}) converges, in some sense, to the solution of the equation (\ref{Apr}).

The paper is organized as follows. 
In section \ref{istrm} we define the martingale solution to some kind of stochastic hybrid Korteweg - de Vries - Burgers equation (\ref{Apr}) with a multiplicative Wiener noise on the interval $[0,T]$. Then we formulate and prove Theorem~\ref{P4.1}. In the proof, some methods introduced in \cite{Gat} and extended in \cite{Deb} have been adapted to the problem considered. 

In section \ref{dow1}  Lemmas \ref{szac4.1} and \ref{PropCias} used in the proof of Theorem \ref{P4.1} are proved. 
Lemma \ref{szac4.1} contains a version of estimates which are analogous to those presented in \cite{Deb} and \cite{Gat}. 

In section \ref{dow2} we give the detailed proof of Lemma \ref{parMart}. This lemma formulates sufficient conditions for the existence of martingale solutions for $m$-dimensional Galerkin approximation of Korteweg - de Vries - Burgers equation with a multiplicative Wiener noise for arbitrary $m\in \mathrm{N}$.  

\section{Existence of martingale solution} \label{istrm}

Denote
 $X:=[x_{1},x_{2}]\subset\mathbb{R}$, where $-\infty<x_{1}<0<x_{2}<\infty$.
We consider the following initial value problem for hybrid 
 Korteweg - de Vries - Burgers type equation 
\begin{equation}\label{Apr}
\begin{cases}
d u(t,x) + \big( Au(t,x)u_{x}(t,x) + Bu_{3x}(t,x) - Cu_{2x}(t,x) + Du(t,x) \big) d t \\  \hspace{8ex} = \Phi\left(u(t,x)\right) d W(t) \\
u(0,x) = u_{0}(x), \quad x\in X, \quad t\ge 0. \\
\end{cases}
\end{equation}
In (\ref{Apr}), $W(t)$, $t\geq 0$, is a  cylindrical Wiener process adapted to the filtration $\left\{\mathscr{F}_{t}\right\}_{t\geq 0}$, $u_{0} \in L^{2}(X)$ is a deterministic real-valued function. In (\ref{Apr}) $u(\omega,\cdot,\cdot):\mathbb{R}_{+}\times\mathbb{R} \rightarrow \mathbb{R}$ for all $\omega\in\Omega$. We assume that there exists such $\lambda_{X}>0$, that  
\begin{equation}\label{warU}
\left|u(t,x)\right|<\lambda_{X} < \infty \quad \mbox{for~all~} t\in\mathbb{R}_{+} \mbox{~and~all~}  x\in X.
\end{equation} 
This assumption reflects the finiteness of the solutions to the deterministic equation (\ref{hkdvb}) on finite interval $X$ (see, e.g., \cite{MAS12,ElkaKVCRK}). 

By $ H^{1}(X), H^{2}(X), H^{s}(X), s<0$ we denote the Sobolev spaces according to definitions in \cite{Adams}.
We assume that $\Phi: H^{2}(X)\rightarrow L_{0}^{2}(L^{2}(X))$ is a continuous mapping which for all $u\in H^{2}(X)$ fulfills the following conditions:
\begin{equation}\label{W1}
\displaystyle\mathop{\exists}_{\kappa_{1}, \kappa_{2}>0} \quad \left\|\Phi(u(x))\right\|_{L_{0}^{2}(L^{2}(X))} \leq \kappa_{1} \min \left\{\left|u(x)\right|^{2}_{L^{2}(X)}, \left|u(x)\right|_{L^{2}(X)}\right\} + \kappa_{2};
\end{equation}
\begin{align}\label{W2}
& \mbox{there~exist~such~functions}~   a,b\in L^{2}(X) ~\mbox{with~compact support,~that~the~mapping} \\ \nonumber
& u \mapsto \left( \Phi(u)a , \Phi(u)b \right)_{L^{2}(X)} ~\mbox{is~continuous~in~} L^{2}(X).
\end{align}

\begin{definition}
We say that the problem (\ref{Apr}) has a \textbf{martingale solution} on the interval $[0,T]$, $0<T<\infty$, if there exists a stochastic basis $(\Omega,\mathscr{F},\left\{\mathscr{F}_{t}\right\}_{t\geq 0}, \mathbb{P}, \left\{W_{t}\right\}_{t\geq 0} )$, where $\left\{W_{t}\right\}_{t\geq 0}$ is a cylindrical Wiener process, and there exists the process $\left\{u(t,x)\right\}_{t\geq 0}$ adapted to the filtration $\left\{\mathscr{F}_{t}\right\}_{t\geq 0}$ with trajectories in the space 
\begin{equation}\nonumber
L^{\infty}(0,T;L^{2}(X))\cap L^{2}(0,T;L^{2}(X))\cap C(0,T;H^{s}(X)), \quad s<0, \quad \mathbb{P} - a.s.,
\end{equation}
such that
\begin{equation}\nonumber
\begin{aligned}
& \left\langle u(t,x); v(x) \right\rangle + \int_{0}^{t} \left\langle Au(t,x)u_{x}(t,x) + Bu_{3x}(t,x) - Cu_{2x}(t,x) + Du(t,x); v(x) \right\rangle d s \\
& = \left\langle u_{0}(x); v(x)\right\rangle + \left\langle \int_{0}^{t} \Phi(u(s,x)) d W(s); v(x) \right\rangle, \quad \mathbb{P}-a.s. ,
\end{aligned}
\end{equation}
for all $t\in[0,T]$ and $v\in H^{1}(X)$.
\end{definition}
In our consideration we shall assume that the coefficients of the equation (\ref{Apr}) fulfill the following condition
\begin{equation} \label{ccond}
 B,C,D\geq 0\qquad \mbox{with} \qquad 3B \geq A + 1.
\end{equation}
The physical sense of the coefficients $A,B,C,D$ and the fact that $A$ can be positive or negative, see, e.g.~\cite{ElkaK,ElkaKVCRK,MAS12} confirms that the condition (\ref{ccond}) admits for a broad class of physically meaningful equations which contains all particular cases listed in section \ref{int1}.

\begin{thm}\label{P4.1}
If the conditions (\ref{warU})-(\ref{ccond}) hold then for  all real-valued $u_{0}\in L^{2}(X)$ and $0<T<\infty$ there exists a martingale solution to  (\ref{Apr}).
\end{thm}

\begin{proof}
Let $\varepsilon>0$. Consider the following auxiliary problem 
\begin{equation}\label{par}
\begin{cases}
d u^{\varepsilon}(t,x) + [\varepsilon u^{\varepsilon}_{4x}(t,x) + A u^{\varepsilon}(t,x)u^{\varepsilon}_{x}(t,x) + Bu^{\varepsilon}_{3x}(t,x)
- C u^{\varepsilon}_{2x}(t,x)\\
 \hspace{12ex} + Du^{\varepsilon}(t,x) ] d t  =  \Phi\left(u^{\varepsilon}(t,x)\right)d W(t) \\
u^{\varepsilon}_{0}(x) = u^{\varepsilon}(0,x), \quad \varepsilon >0 .
\end{cases}
\end{equation}
In the proof of theorem we shall use the following lemmas.
\begin{lemma}\label{parMart}
For any $\varepsilon > 0$ there exists a martingale solution to the problem (\ref{par}) if the conditions (\ref{W1}), (\ref{W2}) and  (\ref{ccond}) hold.
\end{lemma}
\begin{lemma}\label{szac4.1}
There exists $\varepsilon_{0} > 0$, such that 
\begin{eqnarray}
\label{4.1a}
\exists_{C_{1}>0}\forall_{0 < \varepsilon < \varepsilon_{0}}~  \varepsilon\,\mathbb{E}\left( \left|u^{\varepsilon}(t,x)\right|^{2}_{L^{2}(0,T;H^{2}(\mathbb{R}))} \right)  & \leq \tilde{C}_{1},  \\ 
\label{4.1c} 
\forall_{k\in X_{k}}\exists_{C_{2}(k)>0}\forall_{0 < \varepsilon < \varepsilon_{0}}~ \mathbb{E}\left( \left|u^{\varepsilon}(t,x)\right|^{2}_{L^{2}(0,T;H^{1}(-k,k))} \right)  & \leq \tilde{C}_{2}(k), 
\end{eqnarray}
where $X_{k} = \big\{k>0: \left|k\right| \leq \min\left\{-x_{1}, x_{2}\right\}\big\}$.
\end{lemma}

\begin{lemma}\label{PropCias}
Let  $\mathscr{L}(u^{\varepsilon})$ denote the family of distributions of the solutions $u^{\varepsilon}$ to (\ref{par}).
Then the family  $\mathscr{L}(u^{\varepsilon})$ is tight in $L^{2}(0,T;L^{2}(X))\cap C(0,T;H^{-3}(X))$.
\end{lemma}

Now, substitute in Prohorov's theorem (e.g., see Theorem 5.1 in \cite{Bil}), $S:=L^{2}(0,T;L^{2}(X))\cap C(0,T;H^{-3}(X))$ and $\mathscr{K}:=\left\{\mathscr{L}(u^{\varepsilon})\right\}_{\varepsilon>0}$. Since $\mathscr{K}$ is tight in $S$, then it is sequentially compact, so there exists a subsequence of $\left\{\mathscr{L}(u^{ \varepsilon})\right\}_{\varepsilon>0}$ converging to some measure $\mu$ in  $\bar{\mathscr{K}}$.
Because $\left\{\mathscr{L}(u^{ \varepsilon})\right\}_{\varepsilon>0}$ is convergent, then in Skorohod's theorem (e.g., see Theorem 6.7 in \cite{Bil}) one can substitute $\mu_{\varepsilon}:=\left\{\mathscr{L}(u^{ \varepsilon})\right\}_{\varepsilon>0}$ and $\mu:=\lim_{\varepsilon\rightarrow 0}\mu_{\varepsilon}$. Then there exists a space $(\bar{\Omega}, \bar{\mathscr{F}}, \left\{\bar{\mathscr{F}}_{t}\right\}_{t\geq 0}, \bar{\mathbb{P}})$ and random variables with values in $L^{2}(0,T;L^{2}(X))\cap C(0,T;H^{-3}(X))$, such that $\bar{u}^{ \varepsilon}\rightarrow\!\bar{u}$ in $L^{2}(0,T;L^{2}(X))$ and $\bar{u}^{ \varepsilon}\rightarrow\!\bar{u}$ in $C(0,T;H^{-3}(X))$. Moreover, $\mathscr{L}(\bar{u}^{ \varepsilon})\! = \mathscr{L}(u^{ \varepsilon})$. 

Then due to Lemma \ref{szac4.1}, for any $p\in\mathbb{N}$ there exist constants $\tilde{C}_{1}(p)$, $\tilde{C}_{2}$ such that
\begin{equation}\nonumber
\begin{aligned}
\mathbb{E}(\sup_{t\in[0,T]} \left|\bar{u}^{ \varepsilon}(t,x)\right|_{L^{2}(X)}^{2p}) \leq \tilde{C}_{1}(p), \quad
\mathbb{E}(\left|\bar{u}^{ \varepsilon}(t,x)\right|^{2}_{L^{2}(0,T;H^{2}(X))}) \leq \tilde{C}_{2}
\end{aligned}
\end{equation}
and $\bar{u}^{ \varepsilon}(t,x)\in L^{2}(0,T; H^{1}(-k,k))\cap L^{\infty}(0,T; L^{2}(X)),~  \mathbb{P}-a.s.$ Then one can conclude that
 $\bar{u}^{ \varepsilon}\rightarrow \bar{u}$ weakly in $L^{2}(\bar{\Omega},L^{2}(0,T;H^{1}(-k,k)))$. 

 Let $x\in\mathbb{R}$ be fixed. Denote
\begin{equation}\nonumber
\begin{aligned}
M^{ \varepsilon}(t) := & u^{ \varepsilon}(t,x) - u_{0}^{  \varepsilon}(x) + \int_{0}^{t}\bigg[ \varepsilon u^{\varepsilon}(t,x)_{4x}(t,x) + A u^{ \varepsilon}(t,x)u^{ \varepsilon}_{x}(t,x)\\
&\hspace{22ex} + B u^{ \varepsilon}_{3x}(t,x) - C u^{ \varepsilon}_{2x}(t,x) + Du^{\varepsilon}(t,x) \bigg] d s,  \\
\bar{M}^{ \varepsilon}(t) := & \bar{u}^{ \varepsilon} (t,x) - \bar{u}_{0}^{  \varepsilon}(x) + \int_{0}^{t}\bigg[ A \bar{u}^{ \varepsilon}(t,x)\bar{u}^{ \varepsilon}_{x}(t,x) 
+ B \bar{u}^{ \varepsilon}_{3x}(t,x) - C \bar{u}^{ \varepsilon}_{2x}(t,x) + D\bar{u}^{\varepsilon}(t,x) \bigg] d s.  
\end{aligned}
\end{equation}
Note, that
\begin{equation}\nonumber
\begin{aligned}
M^{ \varepsilon}(t) = &\hspace{1ex} u_{0}^{  \varepsilon}(x) - \int_{0}^{t}\bigg[ \varepsilon u^{\varepsilon}(t,x)_{4x}(t,x) + A u^{ \varepsilon}(t,x)u^{ \varepsilon}_{x}(t,x)
+ B u^{ \varepsilon}_{3x}(t,x) - C u^{ \varepsilon}_{2x}(t,x) \\  & \hspace{13ex}
 + Du^{\varepsilon}(t,x) \bigg] d s  
+ \int_{0}^{t} \left( \Phi \left( u^{ \varepsilon} (s,x) \right) \right) d W (s) \\ & - u_{0}^{  \varepsilon}(x) +  \int_{0}^{t}\bigg[ A u^{ \varepsilon}(t,x)u^{ \varepsilon}_{x}(t,x) 
+ B u^{ \varepsilon}_{3x}(t,x) - C u^{ \varepsilon}_{2x}(t,x) + Du^{\varepsilon}(t,x) \bigg] d s  \\= & 
\int_{0}^{t} \left( \Phi \left( u^{  \varepsilon} (s,x) \right) \right) d W (s). 
\end{aligned}
\end{equation}
So, $M^{ \varepsilon}(t)$, $t\geq 0$, is a square integrable martingale with values in 
 $L^{2}(X)$, adapted to the filtration $\sigma\left\{u^{  \varepsilon}(s,x), 0\leq s \leq t\right\}$ with quadratic variation equal 
\begin{equation}\nonumber
\left[M^{ \varepsilon}\right](t) := \int_{0}^{t}\Phi(u^{  \varepsilon}(s,x))\left[\Phi(u^{  \varepsilon}(s,x))\right]^{*} d s .
\end{equation}
Substitute in the Doob inequality (e.g., see Theorem 2.2 in \cite{Gaw})
 $M_{t}:=M^{  \varepsilon}(t)$ and $p:=2p$. Then
\begin{equation}\label{szacDoob}
\mathbb{E}\left[\left( \sup_{t\in[0,T]} \left|M^{  \varepsilon}(t)\right|_{L^{2}(X)}^{p} \right)\right] \leq \left( \frac{p}{p-1} \right)^{p} \mathbb{E} \left(\left|M^{  \varepsilon}(T)\right|_{L^{2}(X)} \right).
\end{equation}

Assume $0\leq s \leq t \leq T$ and let $\varphi$ be a bounded continuous function on $L^{2}(0,s;L^{2}(X))$ or $C(0,s;H^{-3}(X))$. Let $a,b\in H^{3}_{0}(-k,k)$, $k\in\mathbb{N}$, be arbitrary and fixed. Since $M^{ \varepsilon}(t)$ is a martingale and $\mathscr{L}(\bar{u}^{ \varepsilon}) = \mathscr{L}(u^{ \varepsilon})$, then (see \cite{Gat}, p. 377-378)
\begin{equation}\nonumber
\mathbb{E}  \Big( \left\langle M^{ \varepsilon}(t) - M^{ \varepsilon}(s); a \right\rangle \varphi\left(u^{ \varepsilon}(t,x)\right)\Big) = 0 \quad \mbox{and} \quad
\mathbb{E}  \Big( \left\langle \bar{M}^{ \varepsilon}(t) - \bar{M}^{ \varepsilon}(s); a \right\rangle\varphi\left(\bar{u}^{ \varepsilon}(t,x)\right)\Big) = 0 .
\end{equation}
Moreover
\begin{equation}\nonumber
\begin{aligned}
\mathbb{E} & \bigg\{\bigg[\left\langle  M^{ \varepsilon}(t);a \right\rangle \left\langle M^{ \varepsilon}(t);b \right\rangle - \left\langle M^{ \varepsilon}(s);a \right\rangle \left\langle M^{ \varepsilon}(s);b \right\rangle \\
& - \int_{s}^{t} \left\langle \left[\Phi\left(u^{ \varepsilon}(\xi,x)\right)\right]^{*}a ; \left[\Phi\left(u^{ \varepsilon}(\xi,x)\right)\right]^{*}b \right\rangle d \xi\bigg]\varphi(u^{ \varepsilon}(t,x))\bigg\} = 0 \\ \mbox{and}\quad
\mathbb{E} & \bigg\{\bigg[\left\langle  \bar{M}^{ \varepsilon}(t);a \right\rangle \left\langle \bar{M}^{ \varepsilon}(t);b \right\rangle - \left\langle \bar{M}^{ \varepsilon}(s);a \right\rangle \left\langle \bar{M}^{ \varepsilon}(s);b \right\rangle \\
& - \int_{s}^{t} \left\langle \left[\Phi\left(\bar{u}^{ \varepsilon}(\xi,x)\right)\right]^{*}a ; \left[\Phi\left(\bar{u}^{ \varepsilon}(\xi,x)\right)\right]^{*}b \right\rangle d \xi\bigg]\varphi(\bar{u}^{ \varepsilon}(t,x))\bigg\} = 0 .
\end{aligned}
\end{equation}

Denote $\bar{M}(t) := \bar{u}(t,x)-u_{0}(x) + \int_{0}^{t}\bigg[ A \bar{u}(t,x)\bar{u}_{x}(t,x)
+ B\bar{u}_{3x}(t,x) - C \bar{u}_{2x}(t,x) + D\bar{u}(t,x) \bigg] d s$. If $\varepsilon \rightarrow 0$, then
$ \bar{M}^{ \varepsilon}(t) \rightarrow \bar{M}(t)$ and $\bar{M}^{  \varepsilon}(s) \rightarrow \bar{M}(s)$, $\bar{\mathbb{P}}$ - a.s. in $H^{-3}(X)$. Moreover, since $\varphi$ is continuous,
then $\varphi(\bar{u}^{  \varepsilon}(s,x)) \rightarrow \varphi(\bar{u}(s,x))$, $\bar{\mathbb{P}}$ - a.s. So, if $\varepsilon \rightarrow 0$, then
\begin{equation}\nonumber
\begin{aligned}
\mathbb{E}&  \Big( \left\langle \bar{M}^{  \varepsilon}(t) - \bar{M}^{  \varepsilon}(s); a \right\rangle \varphi(\bar{u}^{  \varepsilon}(t,x))\Big)  \rightarrow \mathbb{E} \Big( \left\langle \bar{M}(t) - \bar{M}(s); a \right\rangle \varphi(\bar{u}(t,x))\Big) .
\end{aligned}
\end{equation}

Moreover, since $\Phi$ is a continuous operator in topology $L^{2}(X)$ and (\ref{szacDoob}) holds, so if $\varepsilon \rightarrow 0$, then
\begin{equation}\nonumber
\begin{aligned}
\left\langle \left(\Phi(\bar{u}^{ \varepsilon}(s,x))\right)^{*}a; \left(\Phi(\bar{u}^{ \varepsilon}(s,x))\right)^{*}b\right\rangle \rightarrow \left\langle \left(\Phi(\bar{u}(s,x))\right)^{*}a; \left(\Phi(\bar{u}(s,x))\right)^{*}b \right\rangle
\end{aligned}
\end{equation}
and
\begin{equation}\nonumber
\begin{aligned}
\mathbb{E} & \bigg\{\bigg[\left\langle  \bar{M}^{ \varepsilon}(t);a \right\rangle \left\langle \bar{M}^{ \varepsilon}(t);b \right\rangle - \left\langle \bar{M}^{ \varepsilon}(s);a \right\rangle \left\langle \bar{M}^{ \varepsilon}(s);b \right\rangle \\
& - \int_{s}^{t} \left\langle \left[\Phi\left(\bar{u}^{ \varepsilon}(s,\xi)\right)\right]^{*}a ; \left[\Phi\left(\bar{u}^{ \varepsilon}(s,\xi)\right)\right]^{*}b \right\rangle d \xi\bigg]\varphi(\bar{u}^{ \varepsilon}(t,x))\bigg\} \\
\rightarrow \mathbb{E} & \bigg\{\bigg[\left\langle  \bar{M}(t);a \right\rangle \left\langle \bar{M}(t);b \right\rangle - \left\langle \bar{M}(s);a \right\rangle \left\langle \bar{M}(s);b \right\rangle \\
& - \int_{s}^{t} \left\langle \left[\Phi\left(\bar{u}(s,\xi)\right)\right]^{*}a ; \left[\Phi\left(\bar{u}(s,\xi)\right)\right]^{*}b \right\rangle d \xi\bigg]\varphi(\bar{u}(t,x))\bigg\}.
\end{aligned}
\end{equation}

Then $\bar{M}(t), t\ge 0,$ is also a square integrable martingale adapted to the filtration \\ $\sigma\left\{\bar{u}(s), 0\leq s \leq t\right\}$ with quadratic variation $\int_{0}^{t} \Phi(\bar{u}(s,x))\left(\Phi(\bar{u}(s,x))\right)^{*} d s$. 

Substitute in the representation theorem (e.g., see Theorem 8.2 in \cite{dPZ}), $M_{t}:=\bar{M}(t)$, $[M_{t}]:=\int_{0}^{t} \Phi(\bar{u}(s,x))\times\left(\Phi(\bar{u}(s,x))\right)^{*} d s$ and $\Phi(s):=\Phi(\bar{u}(s,x))$.

Then there exists a process $\tilde{M}(t) = \int_{0}^{t}\Phi(\bar{u}(s,x))d W(s)$,  that $\tilde{M}(t)=\bar{M}(t)$, $\mathbb{\bar{P}}$ - a.s., and	
\begin{equation}\nonumber
\begin{aligned}
& \bar{u}(t,x) - u_{0}(x) + \int_{0}^{t}\bigg[A \bar{u}(t,x)\bar{u}_{x}(t,x) + B\bar{u}_{3x}(t,x) - C \bar{u}_{2x}(t,x) + D\bar{u}(t,x) \bigg] d s \\
&= \int_{0}^{t}\Phi(\bar{u}(s,x))d W(s) .
\end{aligned}
\end{equation}
This implies that
\begin{equation}\nonumber
\begin{aligned}
& \bar{u}(t,x) = u_{0}(x) - \int_{0}^{t}\bigg[A \bar{u}(t,x)\bar{u}_{x}(t,x) + B\bar{u}_{3x}(t,x) - C \bar{u}_{2x}(t,x) + D\bar{u}(t,x) \bigg]  d s \\
& + \int_{0}^{t}\Phi(\bar{u}(s,x))d W(s) , \\
\end{aligned}
\end{equation}
so $\bar{u}(t,x)$ is a solution to (\ref{Apr}), what finishes the proof of Theorem \ref{P4.1} .
\end{proof}

\section{Proofs of Lemma \ref{szac4.1} and Lemma \ref{PropCias}} \label{dow1}

\begin{proof}[Proof of Lemma  \ref{szac4.1}]
Let $p: \mathbb{R} \rightarrow \mathbb{R}$ be a smooth function fulfilling the following conditions
\begin{itemize}
\item[(i)] $p$ is increasing on $X$;
\item[(ii)] $p(x_{1}) = \delta > 0$;
\item[(iii)] $p'(x) > \alpha_{X}$ for all $x\in X$;
\item[(iv)] $Bp'''(x)+Cp''(x)\leq \gamma <-1$ for all $x\in X$. 
\end{itemize}
Additionally, let $F(u^{\varepsilon}) := \int_{X}p(x)(u^{\varepsilon}(x))^{2}d x$. Application of the It\^o formula for $F(u^{\varepsilon})$ yields the formula
\begin{align} \nonumber
d F(u^{\varepsilon}(t,x)) = & \left\langle F'(u^{\varepsilon}(t,x));\Phi(u^{\varepsilon}(t,x)) \right\rangle d W(t) - \left\langle F'(u^{\varepsilon}(t,x));\varepsilon u^{\varepsilon}_{4x} + Au^{\varepsilon}(t,x)u_{x}^{\varepsilon}(t,x) \right. \\
&\left. + B u^{\varepsilon}_{3x}(t,x) - C u^{\varepsilon}_{2x}(t,x) + D u^{\varepsilon}(t,x)\right\rangle d t \\ \nonumber
&+ \frac{1}{2}\text{Tr}\left\{F''(u^{\varepsilon}(t,x))\Phi(u^{\varepsilon}(t,x))\left[\Phi(u^{\varepsilon}(t,x))\right]^{*}\right\}d t ,
\end{align}
where $\left\langle F'(u^{\varepsilon}(t,x));v(t,x) \right\rangle = 2\int_{X}p(x)u^{\varepsilon}(t,x)v(t,x)d x $ and $F'(u^{\varepsilon}(t,x))v(t,x)  = 2p(x)v(t,x) $.

We will use the following auxiliary result.
\begin{lemma}\label{DebL1}[\cite{Deb}, p.242] There exist positive constants $C_{1}, C_{2}, C_{3}$ such that
\begin{equation}\nonumber
\begin{aligned}
\int_{X} p(x)u^{\varepsilon}(t,x)u^{\varepsilon}_{4x}(t,x) d x \geq & \frac{1}{2}\int_{X}p(x)\left[u^{\varepsilon}_{2x}(t,x)\right]^{2}d x - C_{1}\left|u^{\varepsilon}(t,x)\right|^{2}_{L^{2}(X)} \\
&- C_{2}\int_{X}p'(x)\left[u_{x}(t,x)\right]^{2}d x ; \\
\int_{X} p(x)u^{\varepsilon}(t,x)u^{\varepsilon}_{3x}(t,x) d x = & \frac{3}{2}\int_{X}p'(x)\left[u^{\varepsilon}_{x}(t,x)\right]^{2}d x - \frac{1}{2}\int_{X}p'''(x)\left[u(t,x)\right]^{2}d x ; \\
\int_{X} p(x)\left[u^{\varepsilon}(t,x)\right]^{2}u^{\varepsilon}_{x}(t,x) d x \geq & -C_{3}\left(1+\left|u^{\varepsilon}(t,x)\right|_{L^{2}(X)}^{6}\right) - \frac{1}{2}\int_{X}p'(x)\left[u_{x}(t,x)\right]^{2}d x . \\
\end{aligned}
\end{equation}
\end{lemma}

Similarly as in Lemma \ref{DebL1}, one has
\begin{equation}\nonumber
\begin{aligned}
\int_{X} p(x)u^{\varepsilon}(t,x)u^{\varepsilon}_{2x}(t,x) d x = & \frac{1}{2}\int_{X}p''(x)\left[u^{\varepsilon}(t,x)\right]^{2}d x - \int_{X}p(x)\left[u_{x}(t,x)\right]^{2}d x .
\end{aligned}
\end{equation}

These estimations imply
\begin{equation}\nonumber
\begin{aligned}
&\left\langle F'(u^{\varepsilon}(t,x));\varepsilon u_{4x}^{\varepsilon}(t,x) + Au^{\varepsilon}(t,x)u_{x}^{\varepsilon}(t,x) + B u^{\varepsilon}_{3x}(t,x) - C u^{\varepsilon}_{2x}(t,x) + D u^{\varepsilon}(t,x)\right\rangle \\
\geq & \varepsilon\int_{X}p(x)\left[u^{\varepsilon}_{2x}(t,x)\right]^{2}d x - 2\varepsilon C_{1}\left|u^{\varepsilon}(t,x)\right|^{2}_{L^{2}(X)} - 2\varepsilon C_{2}\int_{X}p'(x)\left[u_{x}(t,x)\right]^{2}d x \\
& + 3B\int_{X}p'(x)\left[u^{\varepsilon}_{x}(t,x)\right]^{2}d x - B\int_{X}p'''(x)\left[u^{\varepsilon}(t,x)\right]^{2}d x - 2A C_{3}\left(1+ \left|u^{\varepsilon}(t,x)\right|_{L^{2}(X)}^{6}\right) \\
& - A \int_{X}p'(x)\left[u_{x}(t,x)\right]^{2}d x - C \int_{X}p''(x)\left[u^{\varepsilon}(t,x)\right]^{2}d x + 2 C \int_{X}p(x)\left[u_{x}(t,x)\right]^{2}d x \\
& + 2 D \int_{X}p(x)\left[u(t,x)\right]^{2}d x \\
\geq & \varepsilon\int_{X}p(x)\left[u^{\varepsilon}_{2x}(t,x)\right]^{2}d x - 2\varepsilon C_{1}\left|u^{\varepsilon}(t,x)\right|^{2}_{L^{2}(X)} + \left(3B - A - 2\varepsilon C_{2}\right) C_{2}\int_{X}p'(x)\left[u_{x}(t,x)\right]^{2}d x \\
& - \int_{X}\left[Bp'''(x)+Cp''(x)\right]\left[u^{\varepsilon}(t,x)\right]^{2}d x - 2A C_{3}\left(1+ \left|u^{\varepsilon}(t,x)\right|_{L^{2}(X)}^{6}\right) \\
\geq & \varepsilon\int_{X}p(x)\left[u^{\varepsilon}_{2x}(t,x)\right]^{2}d x + \left(3B - A - 2\varepsilon C_{2}\right) \int_{X}p'(x)\left[u_{x}(t,x)\right]^{2}d x \\
& - \left(\gamma + 2\varepsilon C_{1}\right)\left|u^{\varepsilon}(t,x)\right|^{2}_{L^{2}(X)} - 2A C_{3}\left(1+ \left|u^{\varepsilon}(t,x)\right|_{L^{2}(X)}^{6}\right).
\end{aligned}
\end{equation}
Let $\varepsilon \leq \min\left\{\frac{3B-A-1}{2C_{2}}, -\frac{1 +\gamma}{2C_{1}}\right\}$. Then
\begin{equation}\label{DebL2}
\begin{aligned}
&\left\langle F'(u^{\varepsilon}(t,x));\varepsilon u_{4x}^{\varepsilon}(t,x) + Au^{\varepsilon}(t,x)u_{x}^{\varepsilon}(t,x) + B u^{\varepsilon}_{3x}(t,x) - C u^{\varepsilon}_{2x}(t,x) + D u^{\varepsilon}(t,x)\right\rangle \\
\geq & \varepsilon\int_{X}p(x)\left[u^{\varepsilon}_{2x}(t,x)\right]^{2}d x + \int_{X}p'(x)\left[u_{x}(t,x)\right]^{2}d x + \left|u^{\varepsilon}(t,x)\right|^{2}_{L^{2}(X)}\\
& - 2A C_{3}\left(1+ \left|u^{\varepsilon}(t,x)\right|_{L^{2}(X)}^{6}\right) \\
\geq & \varepsilon\int_{X}p(x)\left[u^{\varepsilon}_{2x}(t,x)\right]^{2}d x + \int_{X}p'(x)\left[u_{x}(t,x)\right]^{2}d x - 2A C_{3}\left(1+ \left|u^{\varepsilon}(t,x)\right|_{L^{2}(X)}^{6}\right).
\end{aligned}
\end{equation}
Let $\left\{e_{i}\right\}_{i\in\mathbb{N}}$ be an orthonormal basis in
 $L^{2}(X)$. Then, there exists a constant $C_{4}>0$ such that
\begin{equation}\label{DebL3}
\begin{aligned}
\text{Tr}\left(F''(u)\Phi(u)\left[\Phi(u)\right]^{*}\right) =& 2\sum_{i\in\mathbb{N}} \int_{X}p(x)\left|\Phi\left(u^{\varepsilon}(t,x)\right)e_{i}(x)\right|^{2} d x \leq C_{4}\left|\Phi\left(u^{\varepsilon}(t,x)\right)\right|^{2}_{L_{0}^{2}\left(L^{2}(X)\right)} \\
\leq & C_{4}\left(\kappa_{1}\left|u^{\varepsilon}(t,x)\right|_{L^{2}(X)}^{2}+\kappa_{2}\right)^{2}.
\end{aligned}
\end{equation}
From (\ref{DebL2}) and (\ref{DebL3}) we have
\begin{equation}\nonumber
\begin{aligned}
\mathbb{E}F(u^{\varepsilon}(t,x)) \leq & F\left(u^{\varepsilon}_{0}\right) - \varepsilon\mathbb{E}\int_{0}^{t}\int_{X}p(x)\left[u^{\varepsilon}_{2x}(t,x)\right]^{2}d xd t - \mathbb{E}\int_{0}^{t}\int_{X}p'(x)\left[u_{x}(t,x)\right]^{2}d xd t\\
& +2AC_{3} \mathbb{E}\int_{0}^{t}\left(1+ \left|u^{\varepsilon}(t,x)\right|_{L^{2}(X)}^{6}\right)d t + C_{4}\mathbb{E}\left(\kappa_{1}\left|u^{\varepsilon}(t,x)\right|^{2}_{L^{2}(\mathbb{R})}+\kappa_{2}\right)^{2}, \\
\end{aligned}
\end{equation}
so
\begin{equation}\nonumber
\begin{aligned}
&\mathbb{E}F(u^{\varepsilon}(t,x)) + \varepsilon\mathbb{E}\int_{0}^{t}\int_{X}p(x)\left[u^{\varepsilon}_{2x}(t,x)\right]^{2}d xd t + \mathbb{E}\int_{0}^{t}\int_{X}p'(x)\left[u_{x}(t,x)\right]^{2}d xd t  \\
\leq & F\left(u^{\varepsilon}_{0}\right) +2AC_{3} \mathbb{E}\int_{0}^{t}\left(1+ \left|u^{\varepsilon}(t,x)\right|_{L^{2}(X)}^{6}\right)d t + C_{4}\mathbb{E}\left(\kappa_{1}\left|u^{\varepsilon}(t,x)\right|^{2}_{L^{2}(\mathbb{R})}+\kappa_{2}\right)^{2} \\
\leq & F\left(u^{\varepsilon}_{0}\right) +2AC_{3} \mathbb{E}\int_{0}^{T}\left(1+ \left|u^{\varepsilon}(t,x)\right|_{L^{2}(X)}^{6}\right)d t + C_{4}\mathbb{E}\left(\kappa_{1}\left|u^{\varepsilon}(t,x)\right|^{2}_{L^{2}(\mathbb{R})}+\kappa_{2}\right)^{2} \\
\leq & F\left(u^{\varepsilon}_{0}\right) +2AC_{3} \mathbb{E}\int_{0}^{T}\left(1+ C_{5}\right)d t + C_{6} = F\left(u^{\varepsilon}_{0}\right) + 2AC_{3}T(1+C_{5}) + C_{6} \leq C_{7} .
\end{aligned}
\end{equation}
Let $\varepsilon_{0}>0$ be fixed. Then for all $0<\varepsilon<\varepsilon_{0}$ one has
\begin{equation}\nonumber
\begin{aligned}
\varepsilon & \mathbb{E}\left( \left|u^{\varepsilon}(t,x)\right|^{2}_{L^{2}(0,T;H^{2}(X))} \right) =  \varepsilon \mathbb{E}  \int_{0}^{T} \int_{X} \left[u^{\varepsilon}(t,x)\right] ^{2} d x d t  + \varepsilon \mathbb{E}  \int_{0}^{T} \int_{X} \left[u_{2x}^{\varepsilon}(t,x)\right] ^{2} d x d t  \\ &
\leq  \varepsilon C_{8} + \varepsilon \mathbb{E}  \int_{0}^{T} \int_{X} \left[u_{2x}^{\varepsilon}(t,x)\right] ^{2} d x d t 
=  \varepsilon C_{8} + \varepsilon \mathbb{E}  \int_{0}^{T} \int_{X}\frac{1}{p(x)} p(x) \left[u_{2x}^{\varepsilon}(t,x)\right] ^{2} d x d t \\ &
\leq \varepsilon C_{8} + \varepsilon \mathbb{E}  \int_{0}^{T} \int_{X}\frac{1}{\delta} p(x) \left[u_{2x}^{\varepsilon}(t,x)\right] ^{2} d x 
\leq  \varepsilon C_{8} + \frac{1}{\delta} \varepsilon  \mathbb{E}  \int_{0}^{T} \int_{X} p(x) \left[u_{2x}^{\varepsilon}(t,x)\right] ^{2} d x \\ &
\leq  \varepsilon C_{8} + \frac{1}{\delta} C_{7} \leq C_{9}(\varepsilon+\frac{1}{\delta}) \leq C_{9}(\varepsilon_{0}+\frac{1}{\delta}),
\end{aligned}
\end{equation}
which proves the formula (\ref{4.1a}). Moreover, we have
\begin{equation}\nonumber
\begin{aligned}
\mathbb{E}&\left( \left|u^{\varepsilon}(t,x)\right|^{2}_{L^{2}(0,T;H^{1}(-k,k))} \right) =  \mathbb{E}  \int_{0}^{T} \int_{-k}^{k} \left[u^{\varepsilon}(t,x)\right] ^{2} d x d t  + \mathbb{E}  \int_{0}^{T} \int_{-k}^{k} \left[u_{x}^{\varepsilon}(t,x)\right] ^{2} d x d t  \\   & 
\leq C_{10} + \mathbb{E}  \int_{0}^{T} \int_{-k}^{k} \left[u_{x}^{\varepsilon}(t,x)\right] ^{2} d x 
\leq C_{10} + \mathbb{E}  \int_{0}^{T} \int_{X} \left[u_{x}^{\varepsilon}(t,x)\right] ^{2} d x \\ &
\leq C_{10} + \mathbb{E}  \int_{0}^{T} \int_{X} \frac{1}{p'(x)} p'(x) \left[u_{x}^{\varepsilon}(t,x)\right] ^{2} d x .
\end{aligned}
\end{equation}
Since $p'(x)$ is bounded from below on every compact set $X$  by positive number $\alpha_{X}>0$, then
\begin{equation}\nonumber
\begin{aligned}
\mathbb{E}\left( \left|u^{\varepsilon}(t,x)\right|^{2}_{L^{2}(0,T;H^{1}(-k,k))} \right)  &
\leq C_{10} + \frac{1}{\alpha_{X}} \mathbb{E}  \int_{0}^{T} \int_{X}  p'(x) \left[u_{x}^{\varepsilon}(t,x)\right] ^{2} d x 
\leq C_{10} + \frac{1}{\alpha_{X}} C_{7} \leq C_{11}. 
\end{aligned}
\end{equation}
This proves inequality (\ref{4.1c}).
\end{proof}

\begin{proof}[Proof of Lemma \ref{PropCias}]
Let $k\in X_{k}$ be arbitrary and fixed and let $0<\varepsilon<\varepsilon_{0}\leq 1$. Then
\begin{equation}\label{dec}
\begin{aligned}
u^{\varepsilon}(t,x) & = u_{0}^{\varepsilon}(x) - \! \int_{0}^{t} \!\bigg[\varepsilon u^{\varepsilon}_{4x}(t,x) + Au^{\varepsilon}(t,x)u^{\varepsilon}_{x}(t,x) + Bu^{\varepsilon}_{3x}(t,x) \\ & \hspace{14ex}
- C u^{\varepsilon}_{2x}(t,x) + Du^{\varepsilon}(t,x) \bigg] d s + \!\int_{0}^{t} \!\left( \Phi(u^{\varepsilon}(s,x)) \right) d W(s) .
\end{aligned}
\end{equation}

Denote
\begin{equation}\nonumber
\begin{aligned}
J_{1} & :=  u_{0}^{\varepsilon}(x) ; & 
J_{2}:= - \varepsilon\int_{0}^{t} u^{\varepsilon}_{4x}(t,x)  d s ; \\
J_{3} & :=- A\int_{0}^{t} u^{\varepsilon}(s,x)u^{\varepsilon}_{x}(s,x) d s ;\quad  &
J_{4} :=  - B\int_{0}^{t} u^{\varepsilon}_{3x}(t,x) d s ; \\
J_{5} & :=  C\int_{0}^{t} u^{\varepsilon}_{2x}(t,x) d s ; &
J_{6} := - D\int_{0}^{t} u^{\varepsilon}(t,x)  d s ; \\
J_{7} & := \int_{0}^{t} \left( \Phi(u^{\varepsilon}(s,x)) \right) d W(s) .
\end{aligned}
\end{equation}

Now, we start estimating the above terms.

From the assumption, ~$\mathbb{E} \left|J_{1}\right|^{2}_{W^{1,2}(0,T;H^{-2}(-k,k))} = C_{1}$, where $C_{1}>0$.

Next, there exists a constant  $C_{2} > 0$, that
\begin{equation}\nonumber
\left| - \varepsilon u^{\varepsilon}_{4x}(t,x)  \right|_{H^{-2}(-k,k)} = 
 \varepsilon \left| u^{\varepsilon}_{4x}(t,x)  \right|_{H^{-2}(-k,k)} 
 \leq C_{2}\varepsilon\left| u^{\varepsilon}(s,x) \right|_{H^{2}(-k,k)}.
\end{equation}
So, due to Lemma \ref{szac4.1}, there exists a constant $C_{3}(k)>0$ such that
\begin{equation}\nonumber
\begin{aligned}
\mathbb{E} & \left| - \varepsilon u^{\varepsilon}_{4x}(t,x) \right|^{2}_{L^{2}(0,T;H^{-2}(-k,k))} =  \mathbb{E} \int_{0}^{T} \left| - \varepsilon u^{\varepsilon}_{4x}(t,x) \right|^{2}_{H^{-2}(-k,k)} d s \\ &
\leq C_{2}^{2}\varepsilon^{2} \mathbb{E} \int_{0}^{T} \left| u^{\varepsilon}(s,x) \right|^{2}_{H^{2}(-k,k)} d s 
\leq  C_{3}(k) .
\end{aligned}
\end{equation}

Therefore we can write
$\hspace{1ex} \mathbb{E} \left|J_{2}\right|^{2}_{W^{1,2}(0,T,H^{-2}(-k,k))} \leq C_{4}(k)$, where $C_{4}(k)>0$.

 Now, we will use the following result from \cite{Deb}. 
\begin{lemma}\label{interp}(\cite{Deb}, p. 243) 
There exists a constant $C_{5}(k)$ such that the following inequality holds
$\quad 
\left|u^{\varepsilon}(s,x)u^{\varepsilon}_{x}(s,x)\right|_{H^{-1}(-k,k)} \leq C_{5}(k)\left|u^{\varepsilon}(s,x)\right|^{\frac{3}{2}}_{L^{2}(-k,k)}\left|u^{\varepsilon}(s,x)\right|^{\frac{1}{2}}_{H^{1}(-k,k)}.
$
\end{lemma}
Due to Lemma \ref{interp} there exist positive constants  $C_6,C_7(k),C_{8}(k)$ such that
\begin{equation}\nonumber
\begin{aligned}
&|-A u^{\varepsilon}(s,x)u^{\varepsilon}_{x}(s,x)|_{H^{-2}(-k,k)} =  A \left| u^{\varepsilon}(s,x)u^{\varepsilon}_{x}(s,x)  \right|_{H^{-2}(-k,k)}
\leq  C_{6} A \left| u^{\varepsilon}(s,x)u^{\varepsilon}_{x}(s,x) \right|_{H^{-1}(-k,k)} \\ &
\leq A C_{7}(k) \left|u^{\varepsilon}(s,x)\right|^{\frac{3}{2}}_{L^{2}(-k,k)}\left|u^{\varepsilon}(s,x)\right|^{\frac{1}{2}}_{H^{1}(-k,k)} \\ & 
\leq  A C_{7}(k) \left|u^{\varepsilon}(s,x)\right|_{L^{2}(-k,k)}\left|u^{\varepsilon}(s,x)\right|^{\frac{1}{2}}_{L^{2}(-k,k)}\left|u^{\varepsilon}(s,x)\right|^{\frac{1}{2}}_{H^{1}(-k,k)} \\ & 
\leq A C_{7}(k) \left[\left(2k\lambda_{X}^{2}\right)^{\frac{1}{2}}\right]\left|u^{\varepsilon}(s,x)\right|^{\frac{1}{2}}_{H^{1}(-k,k)}  
\leq  A C_{8}(k) \lambda_{X}\left|u^{\varepsilon}(s,x)\right|_{H^{1}(-k,k)}  .
\end{aligned}
\end{equation}
Due to Lemma \ref{szac4.1} there exists a constant $C_{9}(k)>0$, that we can write
\begin{equation}\nonumber
\begin{aligned}
\mathbb{E}& \left|- A u^{\varepsilon}(s,x)u^{\varepsilon}_{x}(s,x) \right|^{2}_{L^{2}(0,T;H^{-2}(-k,k))} =  \mathbb{E} \int_{0}^{T}\!\!\! \left| - A u^{\varepsilon}(s,x)u^{\varepsilon}_{x}(s,x) \right|^{2}_{H^{-2}(-k,k)} d s \\ &
\leq A^{2} C_{8}^{2}(k) \lambda_{X}^{2} \mathbb{E} \int_{0}^{T}\!\!\! \left|u^{\varepsilon}(s,x)\right|^{2}_{H^{1}(-k,k)} d s 
= A^{2} C_{8}^{2}(k) \lambda_{X}^{2} \mathbb{E} \left|u^{\varepsilon}(s,x)\right|^{2}_{L^{2}(0,T;H^{1}(-k,k))} 
\leq A^{2} C_{9}(k) \lambda_{X}^{2} .
\end{aligned}
\end{equation}
Threfore, we obtain 
$\hspace{1ex}
\mathbb{E} \left|J_{3}\right|^{2}_{W^{1,2}(0,T,H^{-2}(-k,k))} \leq C_{10}(k)$, 
 where $C_{10}(k)>0$.

 Next, there exist constants $C_{11},C_{12} > 0$, that
\begin{equation}\nonumber
\begin{aligned}
\left| - B u^{\varepsilon}_{3x}(t,x)\right|_{H^{-2}(-k,k)} & =  B \left| u^{\varepsilon}_{3x}(t,x)  \right|_{H^{-2}(-k,k)}
\leq  B C_{11}\left| u^{\varepsilon}(s,x) \right|_{H^{1}(-k,k)} \\ &
\leq  B C_{12}\left| u^{\varepsilon}(s,x) \right|_{H^{2}(-k,k)}.
\end{aligned}
\end{equation}
Lemma \ref{szac4.1} implies the existence of a constant $C_{13}(k)>0$, that we can write the following estimates
\begin{equation}\nonumber
\begin{aligned}
\mathbb{E} & \left|- B u^{\varepsilon}_{3x}(t,x) \right|^{2}_{L^{2}(0,T;H^{-2}(-k,k))} =  \mathbb{E} \int_{0}^{T} \left| - B u^{\varepsilon}_{3x}(t,x) \right|^{2}_{H^{-2}(-k,k)} d s \\ &
\leq B^{2} C^{2}_{12} \mathbb{E} \int_{0}^{T} \left|u^{\varepsilon}(s,x)\right|^{2}_{H^{2}(-k,k)} d s
=  B^{2} C^{2}_{12} \mathbb{E} \left|u^{\varepsilon}(s,x)\right|^{2}_{L^{2}(0,T;H^{2}(-k,k))} \\ &
\leq  B^{2} C^{2}_{12} \mathbb{E} \left|u^{\varepsilon}(s,x)\right|^{2}_{L^{2}(0,T;H^{2}(\mathbb{R}))}
\leq  B^{2} C_{13} .
\end{aligned}
\end{equation}
So, we obtain 
$\hspace{1ex} \mathbb{E} \left|J_{4}\right|^{2}_{W^{1,2}(0,T,H^{-2}(-k,k))} \leq C_{14}$, where $C_{14}>0$.

 For some constant $C_{15} > 0$, we have
\begin{equation}\nonumber
\left| C u^{\varepsilon}_{2x}(t,x)\right|_{H^{-2}(-k,k)} =  C \left| u^{\varepsilon}_{2x}(t,x)  \right|_{H^{-2}(-k,k)}
\leq  C C_{15}\left| u^{\varepsilon}(s,x) \right|_{L^{2}(-k,k)} 
\leq  C C_{16}\left| u^{\varepsilon}(s,x) \right|_{H^{2}(-k,k)}.
\end{equation}
Lemma \ref{szac4.1} implies the existence of a constant $C_{17}(k)>0$ such that
\begin{equation}\nonumber
\begin{aligned}
\mathbb{E} & \left| C u^{\varepsilon}_{2x}(t,x) \right|^{2}_{L^{2}(0,T;H^{-2}(-k,k))} =  \mathbb{E} \int_{0}^{T} \left| C u^{\varepsilon}_{2x}(t,x) \right|^{2}_{H^{-2}(-k,k)} d s \\ &
\leq C^{2} C^{2}_{16} \mathbb{E} \int_{0}^{T} \left|u^{\varepsilon}(s,x)\right|^{2}_{H^{2}(-k,k)} d s 
= C^{2} C^{2}_{16} \mathbb{E} \left|u^{\varepsilon}(s,x)\right|^{2}_{L^{2}(0,T;H^{2}(-k,k))} \\ & 
\leq C^{2} C^{2}_{16} \mathbb{E} \left|u^{\varepsilon}(s,x)\right|^{2}_{L^{2}(0,T;H^{2}(\mathbb{R}))} \leq  C^{2} C_{17} .
\end{aligned}
\end{equation}
Hence, we received 
$\hspace{1ex} \mathbb{E} \left|J_{5}\right|^{2}_{W^{1,2}(0,T,H^{-2}(-k,k))} \leq C_{18}$, where  $C_{18}>0$.

 There exists a constant $C_{19} > 0$ such that
\begin{equation}\nonumber
\left| - D u^{\varepsilon}(t,x)\right|_{H^{-2}(-k,k)} =  D \left| u^{\varepsilon}(t,x)  \right|_{H^{-2}(-k,k)}
\leq  D C_{19} \left| u^{\varepsilon}(s,x) \right|_{H^{2}(-k,k)}.
\end{equation}
Due to Lemma \ref{szac4.1} for some constant  $C_{20}(k)>0$, we obtain
\begin{equation}\nonumber
\begin{aligned}
\mathbb{E}& \left| - D u^{\varepsilon}(t,x) \right|^{2}_{L^{2}(0,T;H^{-2}(-k,k))} =  \mathbb{E} \int_{0}^{T} \left| - D u^{\varepsilon}(t,x) \right|^{2}_{H^{-2}(-k,k)} d s \\ &
\leq D^{2} C^{2}_{19} \mathbb{E} \int_{0}^{T} \left|u^{\varepsilon}(s,x)\right|^{2}_{H^{2}(-k,k)} d s 
=  D^{2} C^{2}_{19} \mathbb{E} \left|u^{\varepsilon}(s,x)\right|^{2}_{L^{2}(0,T;H^{2}(-k,k))} \\ &
\leq D^{2} C^{2}_{19} \mathbb{E} \left|u^{\varepsilon}(s,x)\right|^{2}_{L^{2}(0,T;H^{2}(\mathbb{R}))} \leq D^{2} C_{20} .
\end{aligned}
\end{equation}
This implies  
$\hspace{1ex}\mathbb{E} \left|J_{6}\right|^{2}_{W^{1,2}(0,T,H^{-2}(-k,k))} \leq C_{21}$, where $C_{21}>0$.

 Insert in Lemma 2.1 from \cite{Gat}~ $f(s) := \Phi(u(s,x))$, $K\!=\!H\!=\!L^{2}(X)$. Then $\mathscr{I}(f)(t)\! = \!\int_{0}^{t}\Phi(u(s,x)) d W(s)$ and for all  $p\geq 1$ and $\alpha<\frac{1}{2}$ there exists a constant $C_{22}(p,\alpha)>0$, that
\begin{equation}\nonumber
\mathbb{E}\left|\int_{0}^{t}\Phi(u^{m}(s,x)) d W(s)\right|^{2p}_{W^{\alpha ( p),2p}(0,T;L^{2}(X))} \leq  C_{22}(2p,\alpha) \mathbb{E} \left( \int_{0}^{T} \left|\Phi(u^{m}(s,x))\right|^{2p}_{L_{2}^{0}(L^{2}(X))} d s \right). 
\end{equation}
Therefore, due to the condition (\ref{W1}), we can write 
\begin{equation}\nonumber
\mathbb{E}\left|\int_{0}^{t}\Phi(u^{m}(s,x)) d W(s)\right|^{2p}_{W^{\alpha,2p}(0,T;L^{2}(X))} \leq  C_{23} (p,\alpha) , \mbox{~where~} C_{23}>0.
\end{equation}
Substitution $p:=1$ in the above inequality yields
\begin{equation}\label{It\^o3}
\mathbb{E}\left|J_{7}\right|^{2}_{W^{\alpha,2}(0,T;L^{2}(X))} = \mathbb{E}\left|\int_{0}^{t}\Phi(u(s,x)) d W(s)\right|^{2}_{W^{\alpha,2}(0,T;L^{2}(X))} \leq C_{23}(2,\alpha) = C_{24}(\alpha). 
\end{equation}

Let $\beta\in\left(0,\frac{1}{2}\right)$ and $\alpha\in\left(\beta + \frac{1}{2}, \infty\right)$ be arbitrary fixed. Note, that the following relations hold
\begin{itemize}
\item[]$W^{\alpha,2}(0,T;L^{2}(\mathbb{R})) \subset W^{\alpha,2}(0,T;H^{-2}([-k,k))\quad $ and 
\item[]$W^{1,2}(0,T,H^{-2}(-k,k)) \subset W^{\alpha,2}(0,T,H^{-2}(-k,k))$.
\end{itemize}
Therefore, there exists a constant $C_{25}(\alpha) > 0$, such that
\begin{equation}\nonumber
\begin{aligned}
\mathbb{E} & \left|u^{m}(s,x)\right|_{W^{\alpha,2}(0,T,H^{-2}(-k,k))}^{2}  =  \mathbb{E}\left|\sum_{i=1}^{7} J_{i}\right|_{W^{\alpha,2}(0,T,H^{-2}(-k,k))}^{2} \leq \mathbb{E} \left( \sum_{i=1}^{7} \left|J_{i}\right|_{W^{\alpha,2}(0,T,H^{-2}(-k,k))} \right)^{2} \\
= & \mathbb{E} \left[ \sum_{i=1}^{7} \left|J_{i}\right|^{2}_{W^{\alpha,2}(0,T,H^{-2}(-k,k))} + 2\sum_{i=1}^{6} \sum_{j=i+1}^{7} \left|J_{i}\right|_{W^{\alpha,2}(0,T,H^{-2}(-k,k))}\left|J_{j}\right|_{W^{\alpha,2}(0,T,H^{-2}(-k,k))} \right]\\
\leq & \mathbb{E} \left[ \sum_{i=1}^{7} \left|J_{i}\right|^{2}_{W^{\alpha,2}(0,T,H^{-2}(-k,k))} + 2\sum_{i=1}^{6} \sum_{j=i+1}^{7} \left(\left|J_{i}\right|^{2}_{W^{\alpha,2}(0,T,H^{-2}(-k,k))} + \left|J_{j}\right|^{2}_{W^{\alpha,2}(0,T,H^{-2}(-k,k))}\right) \right]\\
= & \mathbb{E} \left[ 8 \sum_{i=1}^{7} \left|J_{i}\right|^{2}_{W^{\alpha,2}(0,T,H^{-2}(-k,k))}\right]
= 8 \sum_{i=1}^{7} \left[ \mathbb{E} \left|J_{i}\right|^{2}_{W^{\alpha,2}(0,T,H^{-2}(-k,k))} \right]
\leq C_{25}(\alpha) .
\end{aligned}
\end{equation}
Moreover, one has 
\begin{itemize}
\item[]$W^{\alpha,2}(0,T,H^{-2}(-k,k)) \subset C^{\beta}(0,T;H^{-3}_{loc}(-k,k)\quad$ and 
\item[]$W^{\alpha,2}(0,T,H^{-2}(\mathbb{R})) \subset W^{\alpha,2}(0,T,H^{-2}(-k,k)) $.
\end{itemize} 
So, there exist constants $C_{27}(k), C_{28}(k, \alpha) >0$, that
\begin{equation}\label{Cszac}
\begin{aligned}
\mathbb{E}\left|u^{\varepsilon}(s,x)\right|_{C^{\beta}(0,T;H^{-3}(-k,k)}^{2} \leq & C_{26} \mathbb{E}\left|u^{\varepsilon}(s,x)\right|_{W^{\alpha,2}(0,T,H^{-3}(-k,k))}^{2} \leq C_{27}(k,\alpha) \\
\mathbb{E}\left|u^{\varepsilon}(s,x)\right|_{W^{\alpha,2}(0,T,H^{-2}(-k,k))} \leq & C_{28}(k, \alpha).
\end{aligned}
\end{equation}
Let $\eta>0$ be arbitrary fixed. Lemma \ref{szac4.1} implies the existence of a constant $C_{30}(k)>0$,  that
\begin{equation}\label{rwnszac}
\begin{aligned}
\mathbb{E}\left|u^{\varepsilon}(s,x)\right|^{2}_{L^{2}(0,T,H^{-1}(-k,k))} \leq & C_{29}(k)\mathbb{E}\left|u^{\varepsilon}(s,x)\right|^{2}_{L^{2}(0,T,H^{-1}(\mathbb{R}))} \tilde{C}_{2} = C_{30}(k).
\end{aligned}
\end{equation}

Substituting in \cite[Lemma 2.1]{Deb} $\alpha_{k}:=\eta^{-1}2^{k} \left( C_{30}(k) + C_{27}(k,\alpha) + C_{28}(k,\alpha) \right)$ and using Markov inequality \cite[p.~114]{Pap}
for\\ $X := \left|u^{\varepsilon}(s,x)\right|^{2}_{L^{2}(0,T,H^{-1}(-k,k))} + \left|u^{\varepsilon}(s,x)\right|^{2}_{W^{\alpha,2}(0,T,H^{-2}(-k,k))} + \left|u^{\varepsilon}(s,x)\right|_{C^{\beta}(0,T;H^{-3}_{loc}(-k,k)}^{2}$ and \\ $ \varepsilon := \eta^{-1}2^{k} \left( C_{30}(k) + C_{27}(k,\alpha) + C_{28}(k,\alpha) \right)$, one obtains
\begin{equation}\nonumber
\begin{aligned}
\mathbb{P} \Big(u^{\varepsilon} \in A\left(\left\{\alpha _{k} \right\} \right) \Big) & = 
  1 - \mathbb{P} \Big( \left|u^{\varepsilon}(s,x)\right|^{2}_{L^{2}(0,T,H^{-1}(-k,k))} + \left|u^{\varepsilon}(s,x)\right|^{2}_{W^{\alpha,2}(0,T,H^{-2}(-k,k))}   \\ 
&  + \left|u^{\varepsilon}(s,x)\right|_{C^{\beta}(0,T;H^{-3}_{loc}(-k,k))}^{2}  \geq \eta^{-1}2^{k} \left( C_{30}(k) + C_{27}(k,\alpha) 
 + C_{28}(k,\alpha) \right) \Big) \\ & 
=1 - \frac{C_{30}(k) + C_{27}(k,\alpha) + C_{28}(k,\alpha)}{\eta^{-1}2^{k} \left( C_{30}(k) + C_{27}(k,\alpha) + C_{28}(k,\alpha)\right)} 
=  1 - \frac{\eta}{2^{k}} > 1 - \eta .
\end{aligned}
\end{equation}
Let $K$ be a mapping such that for  $\eta>0$,  $K\left( \eta \right) := A\left(\left\{a_{k}^{(\eta)}\right\}\right)$, where $\left\{a_{k}^{(\eta)}\right\}$ is an increasing sequence of positive numbers, which can, but does not have to, depend on  $\eta$. Note, that due to \cite[Lemma 2.1]{Deb}, the set $K(\eta)$ is compact for all  $\eta>0$. Moreover, $\mathbb{P}\left\{K\left( \eta \right)\right\} > 1-\eta$, then the family $\mathscr{L}\left(u^{\varepsilon}\right)$ is tight. 
\end{proof}

\section{Proof of Lemma \ref{parMart}} \label{dow2}
\begin{proof}
Let $\left\{e_{i}\right\}_{i\in\mathbb{N}}$ be an orthonormal basis in space  $L^{2}(X)$. Denote by $P_{m}$, for all $m\in\mathbb{N}$, the orthogonal projection on $Sp(e_{1},...,e_{m})$. Consider finite dimensional Galerkin approximation of the problem (\ref{par}) in space  $P_{m}L^{2}(X)$ in the form
\begin{equation}\label{Galerkin}
\begin{cases}
d u^{m,\varepsilon}(t,x) + \left(\varepsilon \theta\left(\frac{\left|u^{m,\varepsilon}_{4x}(t,x)\right|^{2}}{m}\right) u^{m,\varepsilon}_{4x}(t,x) + A\theta\left(\frac{\left|u^{m,\varepsilon}_{x}(t,x)\right|^{2}}{m} \right) u^{m,\varepsilon}(t,x)u^{m,\varepsilon}_{x}(t,x)\right. \\
\left.+ B\theta\left(\frac{\left|u^{m,\varepsilon}_{3x}(t,x)\right|^{2}}{m} \right)u^{m,\varepsilon}_{3x}(t,x) - C \theta\left(\frac{\left|u^{m,\varepsilon}_{2x}(t,x)\right|^{2}}{m} \right)u^{m,\varepsilon}_{2x}(t,x) + Du^{m,\varepsilon}(t,x)  \right) d t\\
=  P_{m}\Phi\left(u^{m,\varepsilon}(t,x)\right)d W^{m}(t) \\
u^{m,\varepsilon}_{0}(x) = P_{m}u^{\varepsilon}(0,x) ,
\end{cases}
\end{equation}
where $\theta\in C^{\infty}(\mathbb{R})$ fulfills conditions
\begin{equation}
\begin{cases}
\theta(\xi) = 1, \quad &\textrm{when} \quad \xi\in [0,1] \\
\theta(\xi) \in [0,1], \quad &\textrm{when} \quad \xi\in (1,2) \\
\theta(\xi) = 0, \quad &\textrm{when} \quad \xi\in \left.[2,\infty)\right. .
\end{cases}
\end{equation}

Let $m\in\mathbb{N}$ be arbitrary fixed and 
\begin{equation}\nonumber
\begin{aligned}
b(u(t,x)) := &\varepsilon \theta\left(\frac{\left|u^{m,\varepsilon}_{4x}(t,x)\right|^{2}}{m}\right) u^{m,\varepsilon}_{4x}(t,x) + A\theta\left(\frac{\left|u^{m,\varepsilon}_{x}(t,x)\right|^{2}}{m} \right) u^{m,\varepsilon}(t,x)u^{m,\varepsilon}_{x}(t,x) \\
&+ B\theta\left(\frac{\left|u^{m,\varepsilon}_{3x}(t,x)\right|^{2}}{m} \right)u^{m,\varepsilon}_{3x}(t,x) - C \theta\left(\frac{\left|u^{m,\varepsilon}_{2x}(t,x)\right|^{2}}{m} \right)u^{m,\varepsilon}_{2x}(t,x) + Du^{m,\varepsilon}(t,x)  , \\
\sigma (u(t,x)) := & P_{m}\Phi(u^{m,\varepsilon}(t,x)).
\end{aligned}
\end{equation} 
Then
\begin{equation}\nonumber
\begin{aligned}
& \left| b(u(t,x)) \right|_{L^{2}(X)} \leq \varepsilon\left| \theta\left(\frac{\left|u^{m,\varepsilon}_{4x}(t,x)\right|^{2}}{m} \right) u^{m,\varepsilon}(t,x)u^{m,\varepsilon}_{4x}(t,x) \right|_{L^{2}(X)} \\
& + A\left| \theta\left(\frac{\left|u^{m,\varepsilon}_{x}(t,x)\right|^{2}}{m} \right) u^{m,\varepsilon}(t,x)u^{m,\varepsilon}_{x}(t,x) \right|_{L^{2}(X)} 
 + B \left| \theta\left(\frac{\left|u^{m,\varepsilon}_{3x}(t,x)\right|^{2}}{m} \right)u^{m,\varepsilon}_{3x}(t,x) \right|_{L^{2}(X)} \\ &
 + C \left| \theta\left(\frac{\left|u^{m,\varepsilon}_{2x}(t,x)\right|^{2}}{m} \right)u^{m,\varepsilon}_{2x}(t,x) \right|_{L^{2}(X)} 
 + D  \left| u^{m,\varepsilon}(t,x) \right|_{L^{2}(X)} \\ &
= : \varepsilon J_{1} + AJ_{2} + BJ_{3} + CJ_{4} + DJ_{5}.
\end{aligned}
\end{equation}
Note, that
\begin{equation}\nonumber
J_{2} = 
\begin{cases}
0, \quad \textrm{when} \quad  \left|u^{m,\varepsilon}_{x}(t,x)\right| \geq \sqrt{2m} \\
\lambda  \left| u^{m,\varepsilon}(t,x)u^{m,\varepsilon}_{x}(t,x) \right|_{L^{2}(X)}, \quad \textrm{when} \quad \left|u^{m,\varepsilon}_{x}(t,x)\right| \leq \sqrt{2m} ,
\end{cases}
\end{equation}
where $\lambda \in [0,1]$, so
\begin{equation}\nonumber
\begin{aligned}
J_{2} \leq \left| u^{m,\varepsilon}(t,x)u^{m,\varepsilon}_{x}(t,x) \right|_{L^{2}(X)} \leq \sqrt{2m} \left| u^{m,\varepsilon}(t,x) \right|_{L^{2}(X)} .
\end{aligned}
\end{equation}
Similarly,
\begin{equation}\nonumber
J_{1} = 
\begin{cases}
0, \quad \textrm{when} \quad  \left|u^{m,\varepsilon}_{4x}(t,x)\right| \geq \sqrt{2m} \\
\lambda  \left|u^{m,\varepsilon}_{4x}(t,x) \right|_{L^{2}(X)}, \quad \textrm{when} \quad \left|u^{m,\varepsilon}_{4x}(t,x)\right| \leq \sqrt{2m} ,
\end{cases}
\end{equation}

\begin{equation}\nonumber
J_{3} = 
\begin{cases}
0, \quad \textrm{when} \quad  \left|u^{m,\varepsilon}_{3x}(t,x)\right| \geq \sqrt{2m} \\
\lambda  \left|u^{m,\varepsilon}_{3x}(t,x) \right|_{L^{2}(X)}, \quad \textrm{when} \quad \left|u^{m,\varepsilon}_{3x}(t,x)\right| \leq \sqrt{2m} ,
\end{cases}
\end{equation}
and
\begin{equation}\nonumber
J_{4} = 
\begin{cases}
0, \quad \textrm{when} \quad  \left|u^{m,\varepsilon}_{2x}(t,x)\right| \geq \sqrt{2m} \\
\lambda  \left|u^{m,\varepsilon}_{2x}(t,x) \right|_{L^{2}(X)}, \quad \textrm{when} \quad \left|u^{m,\varepsilon}_{2x}(t,x)\right| \leq \sqrt{2m} ,
\end{cases}
\end{equation}
where $\lambda \in [0,1]$, so
\begin{equation}\nonumber
J_{1},J_{3},J_{4} \leq \sqrt{2m}.
\end{equation}
Therefore one gets
\begin{equation}\nonumber
\begin{aligned}
& \left| b( u^{m,\varepsilon}(t,x)) \right|_{L^{2}(X)} = \varepsilon J_{1} + AJ_{2} + BJ_{3} + CJ_{4} + DJ_{5} \\
\leq & \,\varepsilon\sqrt{2m} + A \sqrt{2m} \left| u^{m,\varepsilon}(t,x) \right|_{L^{2}(X)} + B\sqrt{2m} + C\sqrt{2m} + D \left| u^{m,\varepsilon}(t,x) \right|_{L^{2}(X)} \\
= & \left(A \sqrt{2m} + D \right) \left| u^{m,\varepsilon}(t,x) \right|_{L^{2}(X)} + \sqrt{2m}\left(\varepsilon + B + C\right).
\end{aligned}
\end{equation}
Moreover, due to the condition (\ref{W1}), there exist constants $\kappa_{1}, \kappa_{2} > 0$, such that
\begin{equation}\nonumber
\left\|\Phi(u^{m}(t,x))\right\|_{L_{0}^{2}(L^{2}(X))} \leq \kappa_{1} \left|u^{m}(t,x)\right|_{L^{2}(X)} + \kappa_{2} ,
\end{equation}
so 
\begin{equation}\nonumber
\begin{aligned}
& \left| b(u^{m,\varepsilon}(t,x)) \right|_{L^{2}(X)} + \left\|\sigma(u^{m}(t,x))\right\|_{L_{0}^{2}(L^{2}(X))} \\
\leq & \left(A \sqrt{2m} + D \right) \left| u^{m,\varepsilon}(t,x) \right|_{L^{2}(X)} + \sqrt{2m}\left( \varepsilon + B + C \right) + \kappa_{1} \left|u^{m}(t,x)\right|_{L^{2}(X)}+ \kappa_{2} \\
= & \left( A \sqrt{2m} + D + \kappa_{1} \right) \left| u^{m,\varepsilon}(t,x) \right|_{L^{2}(X)} + \sqrt{2m} \left( \varepsilon + B + C \right) + \kappa_{2}. \\
\end{aligned}
\end{equation}
Let $\kappa := \max\left\{\kappa_{1},\kappa_{2}\right\}$ and $\Lambda = \max\left\{A, \varepsilon + B + C \right\}$. Then
\begin{equation}\nonumber
\begin{aligned}
& \left| b(u^{m,\varepsilon}(t,x)) \right|_{L^{2}(X)} + \left\|\sigma(u^{m}(t,x))\right\|_{L_{0}^{2}(L^{2}(X))} \\
\leq & \left( \Lambda\sqrt{2m} + \kappa + D \right) \left| u^{m,\varepsilon}(t,x) \right|_{L^{2}(X)} + \Lambda\sqrt{2m} + \kappa + D \\
= & \left( \Lambda\sqrt{2m} + \kappa + D \right)\left( \left| u^{m,\varepsilon}(t,x) \right|_{L^{2}(X)} + 1 \right).
\end{aligned}
\end{equation}
Therefore, by \cite[Proposition 3.6]{Kar}, and \cite[Proposition 4.6]{Kar}, for all $m\in\mathbb{N}$ there exists a martingale solution to (\ref{Galerkin}). Moreover, applying the same methods as in section \ref{dow1} one can show that
 for all $m$ the following inequalities hold
\begin{eqnarray} \nonumber
\exists_{C_{1}(\varepsilon)>0}\mathbb{E}\left( \left|u^{m,\varepsilon}(t,x)\right|^{2}_{L^{2}(0,T;H^{2}(X))} \right)  & \leq &\tilde{C}_{1}(\varepsilon),  \\ \nonumber
\forall_{k\in X_{k}}\exists_{C_{2}(k,\varepsilon)>0} \mathbb{E}\left( \left|u^{m,\varepsilon}(t,x)\right|^{2}_{L^{2}(0,T;H^{1}(-k,k))} \right)  & \leq &\tilde{C}_{2}(k,\varepsilon) 
\end{eqnarray}
and the family of distributions  $\mathscr{L}(u^{m,\varepsilon})$ is tight in $L^{2}(0,T;L^{2}(X))\cap C(0,T;H^{-3}(X))$.
Then application of the same methods as 
on pages \pageref{P4.1}--\pageref{dow1} leads to proof of the existence of martingale solution to  (\ref{par}) with the conditions (\ref{W1}), (\ref{W2}) and (\ref{ccond}).
\end{proof}

\end{document}